\documentclass[12pt,reqno]{amsart}

\usepackage{blindtext}
\usepackage{geometry}
\usepackage{times}
\usepackage[T1]{fontenc}
\usepackage{mathrsfs}
\usepackage{latexsym}
\usepackage[dvips]{graphics}
\usepackage{epsfig}
\usepackage{blindtext}
\usepackage{hyperref}
\usepackage{xcolor}
\usepackage{gensymb}
\usepackage{hyperref, amsmath, amsthm, amsfonts, amscd, flafter,epsf}
\usepackage{amsmath,amsfonts,amsthm,amssymb,amscd}
\input amssym.def
\input amssym.tex
\usepackage{color}
\usepackage{physics}
\usepackage{listings}
\usepackage{tikz}
\usepackage{subfig}
\usepackage{amssymb}
\usepackage{mathtools}
\usepackage{cleveref}

\newcommand\be{\begin{equation}}
\newcommand\ee{\end{equation}}
\newcommand\bea{\begin{eqnarray}}
\newcommand\eea{\end{eqnarray}}
\newcommand\bi{\begin{itemize}}
\newcommand\ei{\end{itemize}}
\newcommand\ben{\begin{enumerate}}
\newcommand\een{\end{enumerate}}

\newtheorem{thm}{Theorem}[section]
\newtheorem{conj}[thm]{Conjecture}

\newtheorem{lem}[thm]{Lemma}

\newtheorem{defi}[thm]{Definition}

\newtheorem{rmk}[thm]{Remark}

\newcommand{\tbf}[1]{\textbf{#1}}

\newcommand{\mcal}[1]{\mathcal{#1}}

\newcommand{\twocase}[5]{#1 \begin{cases} #2 & \text{#3}\\ #4
&\text{#5} \end{cases}   }

\newcommand{\R}{\ensuremath{\mathbb{R}}}

\newcommand{\Z}{\ensuremath{\mathbb{Z}}}

\newcommand{\N}{\mathbb{N}}

\numberwithin{equation}{section}

\begin{document}

\title{Benford Behavior of a Higher-Dimensional Fragmentation Process}

\author{Irfan Durmi\'{c}}
\email{idurmic@student.jyu.fi} \address{Department of
  Mathematics and Statistics, University of Jyväskylä, Jyväskylä, FI, 40740}
\author{Steven J. Miller}
\email{ sjm1@williams.edu} \address{Department of
  Mathematics and Statistics, Williams College, Williamstown, MA 01267}

\date{\today}

\maketitle

\begin{center}
    \begin{abstract}
Nature and our world have a bias! Roughly $30\%$ of the time the number $1$ occurs as the leading digit in many datasets base $10$. This phenomenon is known as Benford's law and it arrises in diverse fields such as the stock market, optimizing computers, street addresses, Fibonacci numbers, and is often used to detect possible fraud. 
Based on previous work, we know that different forms of a one-dimensional stick fragmentation result in pieces whose lengths follow Benford's Law. We generalize this result and show that this can be extended to any finite-dimensional ``volume''. We further conjecture that even lower-dimensional volumes, under the unrestricted fragmentation process, follow Benford's Law.
    \end{abstract}
\end{center}

\tableofcontents



\section{Overview of Benford's Law and Results}\label{sec:introandhist}
At the beginning of the $20^{\rm{th}}$ century, the astronomer and mathematician Simon Newcomb~\cite{Ne} observed that logarithmic tables at his workplace showed a lot of wear and tear at the early pages, but substantially less on the later pages.
Newcomb deduced that his colleagues had a ``bias" towards numbers starting with the digit $1$. In particular, in a sufficiently large dataset with a variety of scales, the digit $1$ shows up as the first digit roughly $30\%$ of the time, the digit $2$ about $18\%$ of the time, and so on. While he did come up with a mathematical model for this interesting relationship, his work remained mostly unnoticed.

It took another 57 years after Newcomb's discovery for physicist Frank Benford~\cite{Ben} to make the same observation as Newcomb: the first pages of logarithmic tables were used far more than others. As his paper was next to a well read physics article, it had a very different fate and this bias now wears his name. 
\begin{defi}[Weak Benford Behaviour] \label{def:benforiginal}
Let $F_d$ be the probability that, in a dataset, the first digit base $B$ is $d$. If 
\bea
  F_d \ \coloneqq \log_{B}{\frac{d \ + \ 1}{d}}, \label{eqn:benorg}
\eea
for $d \in \{1,2, \dots, B-1 \}$, we say that the dataset exhibits \emph{(weak) Benford behaviour}.
\end{defi}
It is important to note that the probabilities described by Equation~\eqref{eqn:benorg} are irrational and this fact implies that no finite dataset can have proportions that exactly match them. Only in the context of an infinite dataset that perfectly adheres to Benford's Law would we see the proportions of leading digits aligning exactly with these probabilities. In practice, when examining finite datasets, the focus shifts to how closely they approximate Benford's distribution. Typically, ``close" is interpreted in terms of the observed dataset being a good visual fit to the ``ideal'' Benford distribution of Equation~\eqref{eqn:benorg}. Such, and more intricacies related to Benford's law, are explored in more detail in~\cite{Mil}.\\
Given a base $B \geq 2$, any nonzero real number $r$ can be uniquely expressed in the form $r \  =  \ aB^n$, where $\abs{a} \in [1, B), n \in \Z$. This is usually referred to as \tbf{scientific notation}, and it motivates the definition of the \tbf{significand}.

\begin{defi}[Significand] \label{def:significand}
Given a base $B \ > \ 1$, we define the significand as the mapping $S_B: \mathbb R_{\ne 0} \rightarrow [1,B)$, where $S_B(x)$ is $\abs{a}$ when $x \  =  \ aB^n$ written in scientific notation with $|a| \in [1,B)$ and $n \in \mathbb{Z}$. 
\end{defi}

\noindent One also studies the \tbf{mantissa}, which is the fractional part of the logarithm (base $B$). 
\begin{defi}[Strong Benford's Law]\label{defi:benfstrong}
Writing any $x > 0$ in scientific notation as $S_B(x) B^{k(x)}$, 
with significand $S_B(x) \in [1,B)$ and $k(x)$ an integer, the 
dataset is said to satisfy \emph{strong Benford's law} if the 
probability of a significand being at most $x$ is $\log_B(x)$. 
Often, when referring to Benford's Law, authors actually mean 
strong Benford's Law. 
\end{defi}

This observation enabled mathematicians, economists, statisticians and other scientists to find ways to observe deviations from this pattern in real datasets, which is often a warning sign that there might be something off in a certain dataset. In particular, the IRS has shown substantial interest in Benford's Law because it enables them to detect \textit{potential} cases of fraud. In a world with finite resources, the ability to notice where it might be reasonable to invest our time and energy is valuable. With a multitude of applications, it is natural to wonder how prevalent a phenomenon like this is which motivated earlier work into decompositions and fragmentation processes. We continue this work in our paper.

In their paper on Benford's Law and Continuous Dependent Random Variables, Becker et al.~\cite{Beck}, amongst many other results, demonstrated that the lengths of stick pieces, created under an unrestricted stick fragmentation process, would follow the Benford distribution. In particular, they started with a stick of fixed length $L$ and they broke up this stick in a particular pattern. One chooses a proportion $p_i \in (0,1)$, where $i$ denotes the $i$-th cutting proportion used, and breaks a stick piece into two new pieces. For example, we cut the first piece of length $L$ with $p_1$ and get two new pieces of length $L(1-p_1)$ and $Lp_1$. Repeating this process $N$ times, we observe that, as $N \to \infty$, the different lengths follow Benford's Law.
This result inspired us to wonder whether or not this extends naturally into higher dimensions, and if we can analyse lower-dimensional quantities in higher dimensions to see whether or not those lower-dimensional quantities also exhibit Benford behavior.
As a concrete example, we can show that, under the unrestricted fragmentation model, the volume of a three-dimensional box is Benford. An interesting question is to check whether or not the areas of the new pieces also follow Benford's Law. We also look at the edges of the created boxes and see whether or not the lengths of those edges also follow Benford's Law.\par
Our main result is Theorem~\ref{thm:unrestrictedmdim}, which states that an unrestricted fragmentation process of an $m$-dimensional cube results in pieces whose $m$-dimensional volumes follow the Benford distribution. This theorem is a natural extension of Theorem $1.5$ of~\cite{Beck}. We go deeply into the details of this fragmentation process in Section~\ref{sec:proofofthm}, but we first present an overview of the key machinery we need. A visual presentation of this process, for $m \ = \ 3$, is given in Figures~\labelcref{fig:cubeabc,fig:3dfragmentationunrestricted}.
\begin{defi}[Mellin Transform, $\mcal{M}_f(s)$]\label{defi:mellintransform}
Let $f(t)$ be a continuous real-valued function on $[0,\infty)$. We define its Mellin transform, denoted $\mcal{M}_f(s)$, as follows:
\bea \label{eqn:mellintransdef}
\mcal{M}_f(s) \ := \ \int_0^\infty f(t) t^s \frac{\dd{t}}{t}. 
\eea
\end{defi}

\begin{defi}[Mellin Transform Condition]
We say a set of functions $\mcal{F}$ satisfies our Mellin transform condition if for any set of $f_u \in  \mcal{F}$ (not necessarily distinct) we have
\bea\label{eqn:mellincondition} 
\lim_{N\to\infty} \sum_{\ell = -\infty \atop \ell \neq 0}^\infty \abs{\prod_{u=1}^{Nm} \mathcal{M}_{f_u}\left(1-\frac{2\pi i \ell}{\log 10}\right)} \ = \ 0.
\eea
\end{defi}

Since $\mcal{M}_f(1) \ = \ 1$, as $f$ is a density, Equation~\eqref{eqn:mellincondition} states that the sum of the absolute values of the other terms tends to zero; note that the Mellin transform of a convolution is the product of the Mellin transforms, and thus Equation~\eqref{eqn:mellincondition} can be viewed as a density arising in our fragmentation.

\begin{defi}[Significand Indicator Function, $\varphi_s$] \label{defi:indicatorfunctionforsign}
For $s \in [1,10)$, let 
\bea\label{def:phissigindic} 
\twocase{\varphi_s(x) \ := \ }{1}{{\rm if\ the\ significand\ of\ $x$\ is\ at\ most\ $s$}}{0}{{\rm otherwise;}}
\eea
thus $\varphi_s$ is the indicator function of the event of a significand at most $s$.
\end{defi}
\begin{defi}[Density Cut]\label{defi:densitycut}
Let $f(t)$ be a probability density function that satisfies the Mellin transform condition as expressed in Equation~\eqref{eqn:mellincondition}. A density cut is then defined to be a random variable $P$ taken from the density $f(t)$. The actual values of the random variable $P$, i.e., when we choose  values for the density cuts, are denoted by $p_i$, where $i$ denotes the $i$-th density cut (cutting proportion) used.
\end{defi}
From now on, the terms ``cutting proportion'' and ``density cut'' will mean the same, and we always appeal to Definition~\ref{defi:densitycut} when talking about density cuts. Furthermore, for the purposes of this paper, we take our density cuts from the uniform distribution on $(0,1)$, but we could have used any distribution satisfying Equation~\eqref{eqn:mellincondition}. The details about how to deal with a more general probability density function mostly stem from the key ideas in Lemma~\ref{rmk:steve'sideaforsymmetry}.
\begin{rmk}\label{eqn:howtoviewdensitycuts}
We view density cuts as random numbers generated between $0$ and $1$.
\end{rmk}
\begin{thm}[Unrestricted $m$-dimensional Fragmentation Model]\label{thm:unrestrictedmdim} 
Fix a continuous probability density $f:(0,1)\to\R$ such that
\bea 
\lim_{N\to\infty} \sum_{\ell = -\infty \atop \ell \neq 0}^\infty \abs{\prod_{u=1}^{Nm} \mathcal{M}_{f_u}\left(1-\frac{2\pi i \ell}{\log 10}\right)} \ = \ 0,
\eea
where each $f_u$ is either $f(t)$ or $f(1-t)$ (the density of $1-P$ if $P$ has 
density $f$).  Given an $m$-dimensional box of $m$-dimensional volume $V$, 
we independently choose density cuts $p_1, p_2, \dots,p_{Nm \ - \ 1}, 
p_{Nm}$ from the unit interval stemming from the probability density function $f$ and the associated random variable $P$. 
After $N$ iterations we have
\bea \label{eqn:volumespieces}
V_1 &\ =\ & Vp_1p_2p_4\cdots p_{2^{Nm -2}} \cdot p_{2^{Nm - 1}}, \dots,  \nonumber\\
V_{(2^m)^N} &\ = \ &   V(1-p_1)(1-p_3)(1-p_7) \cdots (1-p_{2^{Nm - 1}-1})  (1-p_{2^{Nm}-1}) ,
\eea
and

\bea 
P_N(s) \ := \  \frac{ \sum_{i = 1}^{(2^m)^N} \varphi_s(V_i)}{(2^m)^N}
\eea 
is the fraction of partition pieces $V_1, \dots, V_{(2^m)^N}$ with a significand that is at most $s$ (see \eqref{def:phissigindic} for the definition of $\varphi_s$).  
Then

\begin{enumerate}

\item $\lim _{N \to \infty} \mathbb{E} [P_N(s)] = \log_{10}s$, and \label{expectation}

\item $\lim _{N \to \infty} {\rm Var}\left(P_N(s)\right) = 0$. \label{variance}

\end{enumerate}
\end{thm}
\begin{lem}\label{rmk:steve'sideaforsymmetry}
Without loss of generality, we can assume that each density $f_u(x)$ in the Mellin transform condition is symmetric around $x \ = \ 1/2$.
\end{lem}

\begin{proof}

Let $V_j$ be some arbitrary volume piece with $1 \ \leq \ j \ \leq \ 2^{Nm}$ that we cut with a density cut $p_l$, where $1 \ \leq \ l \ \leq \ 2^{Nm}-1$. Without loss of generality, we assume that $p_l  \ < \ 1/2$ and that we get two new volume pieces $V_jp_l$ and $V_j(1-p_l)$. Observe that, in terms of the fragmentation process, this is identical to cutting with a density cut $1-p_l$ because we would get the same pieces: $V_j(1-p_l)$ and $V_j(1-(1-p_l)) \ = \ V_jp_l$. Hence we may assume that $0 \ < \ p_l \ < \ 1/2$ for all possible choices of $l$. Furthermore, if we have some density $f(x)$, we can turn it into a density that is symmetric around $1/2$ by replacing it with a density $g(x) \ = \ \frac{f(x) + f(1-x)}{2}$.
\end{proof}
\noindent One of our goals is to think about all of the assumptions that are necessary for Benford behavior to emerge. As an example, we might wonder how reducing the number of new volume pieces might impact the convergence to Benford behavior. We conjecture that, if we decide to fragment only one piece at a time, we still converge to Benford behavior, which would be a direct generalization of Theorem $1.9.$ of~\cite{Beck}.

The next step would be to restrict ourselves to a fragmentation process where every piece is cut with exactly the same density cut $p$. Using the methods of this paper, one could extend Theorem $1.11.$ of~\cite{Beck} to make it work for an $m$-dimensional fragmentation process for almost all $p$.

Finally, it would be interesting to see what happens when we look at the behavior of the significand of lower-dimensional volumes. We present our conjecture about the Benford behavior of lower-dimensional volumes resulting from an unrestricted fragmentation process. In particular, we present our conjecture in the context of a rectangle due to the promising results of the simulations we have run.

\begin{conj}\label{conj:benfrectangle}
The perimeters of the sub-rectangles from a rectangle undergoing an unrestricted fragmentation process converges to Benford behavior.
\end{conj}

More generally, we have the following.

\begin{conj}\label{conj:benfconjddimlessthanmdim}
A $d$-dimensional volume for an $m$-dimensional object, where $d \ < \ m$, undergoing a fragmentation process follows the Benford distribution for the appropriate set of conditions. 
\end{conj}
In this context, an appropriate set of conditions would be something akin to the Mellin transform condition we use throughout this paper. For more details on progress made in this field, read~\cite{Betti}.

In Section~\ref{sec:mathmethods} we present an overview of the key mathemathical methods used in this paper, after which we prove Theorem~\ref{thm:unrestrictedmdim} in Section~\ref{sec:proofofthm}.


\section{Mathemathical Methods used in the Analysis of Benford Behavior}\label{sec:mathmethods}

In this section, we present all of the relevant mathematical machinery necessary to prove our claims. Many of these theorems and definitions are taken from~\cite{GS,Mil,MT-B} and previous papers in the field such as~\cite{Ben,JKKKM,Jing}. For completeness we state some standard definitions. 




\begin{defi}[Probability Density Function (pdf)]\label{defi:pdf}
The probability density function (pdf), also referred to as the \tbf{density}, of a continuous random variable is a function that describes the probability distribution of the variable's values. A probability density function is non-negative, and integrates to one over the space.
\end{defi}

\begin{defi}[Cumulative Distribution Function (cdf)]\label{defi:cdf}
The cumulative distribution function (cdf) describes the probability that a real-valued random variable $\rm{X}$ with a given pdf is at most $x$. If $f(x)$ is the density of $\rm{X}$, then $\int_{-\infty}^{x}f(t)dt$ gives the cumulative distribution function. 
\end{defi}

\begin{defi}[Big-Oh Notation]\label{defi:bigoh}
By $A(x) \ = \ O(B(x))$, which we read as ``$A(x)$ is of order (or big-Oh) $B(x)$'', we mean that there exits some $C \ > \ 0$ and an $x_0$ such that $\forall x \ \geq \ x_0$, $\abs{A(x)} \ \leq \ C B(x)$. We also write this as $A(x) \ll B(x)$ or $B(x) \gg A(x)$.
\end{defi}

\begin{defi}[Little-Oh Notation]\label{defi:littleoh}
By $A(x) \ = \ o(B(x))$, which we read as ``$A(x)$ is little-Oh of $B(x)$'', we mean that $\lim_{x \to \infty} A(x)/B(x) \ = \ 0$.
\end{defi}


\begin{defi}[Equidistributed]\label{defi:equidistributed}
A sequence $\{x_n\}_{n=-\infty}^{\infty}$, $x_n \in [0,1]$, is equidistributed in $[0,1]$ if
\bea
\lim_{N \to \infty} \frac{\#\{n: \abs{n} \ \leq N, x_n \in [a,b]\}}{2N + 1} \ = \ \lim_{N \to \infty} \frac{\sum_{n=-N}^{N}\xi_{(a,b)}(x_n)}{2N+1}
\eea
for all $(a,b) \subset [0,1]$, where $\xi_{(a,b)}$ is the indicator function on the interval $[a,b]$.
\end{defi}
The following result is well known; see for example~\cite{Dia,MT-B}.
\begin{thm}\label{thm:equimod1benf}
If $y_n = \log_{B}{x_n}$ is equdistributed $\bmod \ 1$ then $x_n$ is Benford base $B$.
\end{thm}

\begin{thm}[Poisson Summation] \label{thm:poissum}
Let $f$, $f'$ and $f''$ be continuous functions which eventually decay at least as fast as $x^{-(1+\eta)}$ for some $\eta > 0$. Then
\bea \label{eqn:poissum}
\sum_{n=-\infty}^{+\infty} f(n) & \ = \ & \sum_{n=-\infty}^{+\infty} \hat{f}(n),
\eea
where $\hat{f}(y) \ = \ \int_{-\infty}^{+\infty}f(x)e^{-2\pi ixy}dx$ is the Fourier transformation of $f$.
\end{thm}

\begin{defi}[Standard Convolution]\label{defi:convolution}
The additive \tbf{convolution} of independent continuous random variables $X$ and $Y$ on $\R$ with densities $f_X$ and $f_Y$ is denoted $f_X*f_Y$, and is given by
\bea\label{eqn:convolutioncont}
(f_X*f_Y)(z) \ \coloneqq \ \int_{-\infty}^{+\infty}f_X(t)f_Y(z-t)\dd t.
\eea
If $X$ and $Y$ are discrete, we have
\bea\label{eqn:convolutiondis}
(f_X*f_Y)(z) \ \coloneqq \ \sum_{n}f_X(t)f_Y(z-t);
\eea
note, of course, that $f_Y(z-x_n)$ is zero unless $z-x_n$ is one of the values where $Y$ has positive probability. 
\end{defi}

\begin{thm}[Convolutions and the Fourier Transform]\label{thm:fourierconvolution}
Let $f,g$ be continuous functions on $\R$. If $\int_{-\infty}^{+\infty}\abs{f(x)}^2$ and $\int_{-\infty}^{+\infty}\abs{g(x)}^2$ are finite, then $h \ = \ f*g$ exists, and $\hat{h}(y) \ = \ \hat{f}(y)\hat{g}(y)$.
\end{thm}




In our case, it is more useful to think about a modified definition of the convolution which is more applicable to products. 

\begin{defi}[Convolution]\label{defi:convolution}
The multiplicative convolution of two functions $f$ and $g$ is
\bea\label{eqn:convolution}
(f*g)(x) \ \coloneqq \ \int^{\infty}_{0}g\left(\frac{x}{t}\right)f(t) \frac{dt}{t} \ = \ \int^{\infty}_{0}f\left(\frac{x}{t}\right)g(t) \frac{dt}{t}.\nonumber
\eea
\end{defi}

The power of Definition~\ref{defi:benfstrong} is that it enables us to encode a lot of information about the distribution by just looking at the significand. This, along with some conditions explored in the work of Jang et al. in~\cite{JKKKM}, give us the tools we need to attack this problem.

Throughout this paper, we denote the set of volume pieces stemming from the fragmentation processes as $\{V_i\}$. Hence, we define the volume proportion. 

\begin{defi}[Volume Proportions, $P_N$] \label{defi:volumeproportions}
Given a collection of $m$-dimensional volumes $\{V_i\}$, the proportion whose significand is at most $s$, $P_N(s)$, is
\bea
P_N(s)\ :=\ \frac{\sum\limits_{\textnormal{over all} \ i} \varphi_s(V_i)}{\#\{V_i\}},
\eea
where $\#\{V_i\}$ denotes the size of our set of volumes (i.e., the number of volume pieces we are working with) and $N$ denotes how far along we are in the process.
\end{defi}

We also crucially appeal to Theorem $1.1$ and Remark $2.3$ of~\cite{JKKKM} which, for the reader's convenience, we restate below as Theorem~\ref{thm:jkkkmmain}.

\begin{thm}[Convergence to Benford for ``nice'' functions]\label{thm:jkkkmmain} 
Let $\{\mathcal{D}_i(\theta)\}_{i\in I}$ be a collection of one-parameter distributions with associated densities $f_{\mathcal{D}_i(\theta)}$ which vanish outside of $[0,\infty)$. Let $p:\N \to I$, $X_1 \sim \mathcal{D}_{p(1)}(1)$, $X_m \sim \mathcal{D}_{p(m)}(X_{m-1})$, and assume \ben 
\item for each $m \ge 2$, 
\bea\label{eq:keydensitylink} 
f_m(x_m) \ = \ \int_0^\infty f_{\mathcal{D}_{p(m)}(1)}\left(\frac{x_m}{x_{m-1}}\right) f_{m-1}(x_{m-1}) \frac{dx_{m-1}}{x_{m-1}} 
\eea where $f_m$ is the density of the random variable $X_m$, and 

\item we have 
\be\label{eq:summeltransformsnozero} 
\lim_{n\to\infty} \sum_{\ell = -\infty \atop \ell \neq 0}^\infty \prod_{m=1}^n (\mathcal{M} f_{\mathcal{D}_{p(m)}(1)})\left(1-\frac{2\pi i \ell}{\log B}\right) \ = \ 0. 
\ee 
\een 
Then as $n\to\infty$ the distribution of leading digits of $X_n$ tends to Benford's law. Further, the error is a nice function of the Mellin transforms. Explicitly, if $Y_n = \log_B X_n$, then 
\bea & & 
\left|{\rm Prob}(Y_n \bmod 1 \in [a,b]) - (b-a)\right| \nonumber\\ 
& & \ \ \ \ \ \ \ \le \ (b-a) \cdot \left|\sum_{\ell = -\infty \atop \ell \neq 0}^\infty \prod_{m=1}^n (\mathcal{M} f_{\mathcal{D}_{p(m)}(1)})\left(1-\frac{2\pi i \ell}{\log B}\right)\right|. \eea 
If $I$ is finite and all densities are continuous, then the second condition holds.
\end{thm}

We use a slightly different version of this theorem which can be found in~\cite{BeckA}, but it can also be readily rederived using the tools we have presented in this section. For the reader's convinience, we present it below. 

\begin{thm}[Adaptation of Theorem 1.1. from~\cite{JKKKM}]\label{thm:productisbenfofnindep}
Let $K_1, \dots, K_N$ be independent random variables with densities $f_{k}$. Assume 
\bea\label{eqn:mellintransappendix}
\lim_{N\to\infty} \sum_{\ell = -\infty \atop \ell \neq 0}^\infty \prod_{k=1}^N \abs{\mathcal{M}f_{k}\left(1-\frac{2\pi i \ell}{\log 10}\right)} \ = \ 0.
\eea
Then as $N \to \infty$, $x_N \ = \ K_1\cdots K_N$ converges to Benford's Law. In particular, if $y_N \ = \ \log_{B} x_N$ then 
\bea & & \label{eqn:boundcondition}
\left|{\rm Prob}(y_n \bmod 1 \in [a,b]) - (b-a)\right| \nonumber\\ 
& & \ \ \ \ \ \ \ \le \ (b-a) \cdot \abs{\sum_{\ell = -\infty \atop \ell \neq 0}^\infty \prod_{k=1}^N \mathcal{M} f_{k}\left(1-\frac{2\pi i \ell}{\log B}\right)}. \eea 
\end{thm}

Even though Theorem~\ref{thm:jkkkmmain} is fairly powerful with a relatively weak condition satisfied by most functions, there still are some possible subtle issues when it comes to convergence which is why we will often appeal to Lemma~\ref{rmk:steve'sideaforsymmetry}.


\section{Proof of Theorem~\ref{thm:unrestrictedmdim}}\label{sec:proofofthm}
In this section, we closely follow the proof from~\cite{Beck} for the $1$-dimensional case, while making the appropriate adjustments along the way in order to complete the generalization for the $m$-dimensional case.
We first present a concrete example for the first iteration of the fragmentation of a $3$-dimensional box, after which we extend our ideas to an $m$-dimensional box.

In Figure~\ref{fig:cubeabc} we show the cube of initial volume $V \ = \ abc$ and then in Figure~\ref{fig:3dfragmentationunrestricted} we show the fragmentation process for the first iteration. 

\begin{figure}[h]
    \centering
    \tikzset{every picture/.style={line width=0.75pt}} 

\begin{tikzpicture}[x=0.75pt,y=0.75pt,yscale=-1,xscale=1]

\draw  [color={rgb, 255:red, 74; green, 144; blue, 226 }  ,draw opacity=1 ][fill={rgb, 255:red, 80; green, 227; blue, 194 }  ,fill opacity=1 ] (122,104.2) -- (165.2,61) -- (375.8,61) -- (375.8,161.8) -- (332.6,205) -- (122,205) -- cycle ; \draw  [color={rgb, 255:red, 74; green, 144; blue, 226 }  ,draw opacity=1 ] (375.8,61) -- (332.6,104.2) -- (122,104.2) ; \draw  [color={rgb, 255:red, 74; green, 144; blue, 226 }  ,draw opacity=1 ] (332.6,104.2) -- (332.6,205) ;

\draw (213,214.4) node [anchor=north west][inner sep=0.75pt]    {$a$};
\draw (102,142.4) node [anchor=north west][inner sep=0.75pt]    {$b$};
\draw (364,180.4) node [anchor=north west][inner sep=0.75pt]    {$c$};

\end{tikzpicture}
    \caption{The box we are fragmenting.}
    \label{fig:cubeabc}
\end{figure}
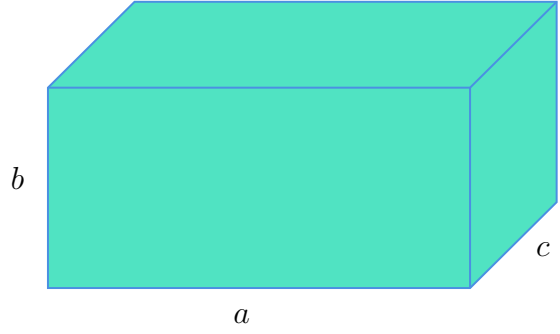

\begin{figure}[h]
    \centering
    \scalebox{1.2}{\includegraphics{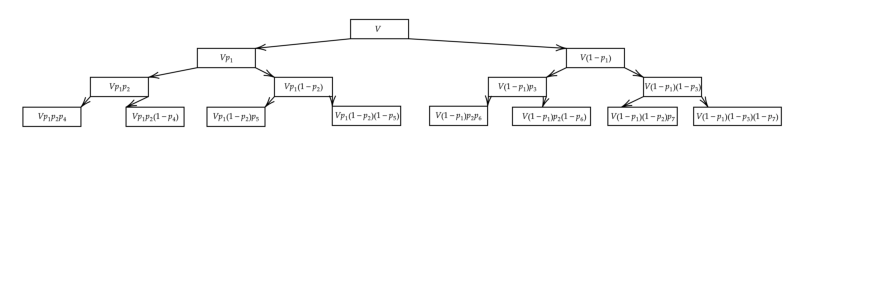}}
    \vspace{-40mm}
    \caption{Fragmentation process for $m \ = \ 3$ and $N \ = \ 1$.}
    \label{fig:3dfragmentationunrestricted}
\end{figure}

Focus now on Figure~\ref{fig:3dfragmentationunrestricted} which represents the fragmentation process. In the first row, we are presenting the initial volume $V \ = \ abc$. We use a density cut $p_1$, taken to be between $0$ and $1$, to cut the box along one of its axes. Without loss of generality, assume we cut the box with a plane vertically (hence, parallel to $b$). Then, we take the leftover pieces, represented in row $2$, and we cut one of those pieces using the density cut $p_2$ and the other one with $p_3$, which, without loss of generality, represent a horizontal cut (parallel to side $a$). The resulting pieces and their volumes are represented in row $3$. Finally, we cut along the last possible direction, with a plane parallel to side $c$ using density cuts $p_4$, $p_5$, $p_6$ and $p_7$. The pieces that result from this process are shown in row $4$ of Figure~\ref{fig:3dfragmentationunrestricted}, and we say we have gone through a single iteration of the process when all of the cuts have been made. In the next iteration, we continue doing the exact same process, so we have $(2^3)^N$ pieces after $N$ iterations. The key insight here is that it doesn't matter in what order we choose to perform the cuts: the only relevant detail, for the way we are constructing this iteration process, is that we cut along each coordinate direction exactly once per iteration. 

\begin{rmk}\label{rmk:whywecutthisway}
We are constructing the process in this way as it leads to a cleaner proof. While we could, for example, assign a particular density cut to every row of the process, we would still get Benford behavior, except in some potential pathological cases, though the algebra would be more complicated. We could even use the same density cut for every piece, and still get Benford behavior as long as $p \  \neq 0.5$. For more details about the machinery of such a process, see the proof of Theorem 1.11 in~\cite{Beck}.
\end{rmk}
After $N$ iterations, we label the pieces as \bea\label{eqn:Nthiterationpiecesform=3}
V_1 &\ =\ & Vp_1p_2p_4\cdots p_{2^{3N -2}} \cdot p_{2^{3N - 1}} \ \textnormal{which is the leftmost piece} \nonumber \\
V_2 &\ =\ & Vp_1p_2p_4\cdots p_{2^{3N -2}} \cdot (1-p_{2^{3N - 1}}) \\
\vdots \nonumber \\
V_{8^N} &\ =\ & V(1-p_1)(1-p_3)(1-p_7)\cdots (1-p_{2^{3N -1}-1}) \cdot (1- p_{2^{3N}-1}),
\eea
where $V_{8^N}$ represents the rightmost volume piece. We use this case as a springboard towards the generalized $m$-dimensional case.

Denote all the coordinates of the $m$-dimensional space as $x_1, x_2, \dots, x_m$. The first step of our first iteration is to cut along the $x_1$ coordinate with some sort of a proportion $p_{(x_1)_{1}} \ = \ p_1$, where $p_{(x_1)_{1}}$ denotes the cut along the $x_1$ coordinate performed during the first iteration and applying it to the ``leftmost'' volume piece in the first iteration. Going back to Figure~\ref{fig:3dfragmentationunrestricted}, this means we cut our original volume $V$ with the density cut $p_1$. This results in two pieces $Vp_1$ and $V(1-p_1)$ and we are done cutting along the $x_1$ direction. We then ``move into the next row'', so this will be the second row in Figure~\ref{fig:3dfragmentationunrestricted} if we denote the starting condition as the first row, and we cut the pieces from left to right along the $x_2$ direction. Again, in our particular case, this means the piece $Vp_1$ is cut with a density cut $p_2$ along the $x_2$ axis, whereas the $V(1-p_1)$ piece is cut with a density cut $p_3$ along the same $x_2$ axis and so on. 
Each density cut $p_i$ is taken to be the value of the random variable $P$ stemming from a probability density function $f(t)$, where $t$ is a dummy variable. This function can be any density function that satisfies the Mellin Transform condition stated in Theorem~\ref{thm:unrestrictedmdim}, which turns out to not be a strong restriction since many functions satisfy it (see~\cite{JKKKM} for details and Example $2.4$ in~\cite{MN1} for a situation where the product is non-Benford). In order to keep the calculations simple, we restrict ourselves to the uniform density on $(0,1)$ while also making comments about the general case along the way. Again, Lemma~\ref{rmk:steve'sideaforsymmetry} makes generalizing the case much easier through appealing to the idea of symmetry around $1/2$.

We keep performing the cuts along every coordinate until we reach the $x_m$ coordinate and we cut along that coordinate using the proportion $p_{(x_m)_{1}} \ = \ p_{2^{m}-1}$. 

After we have completed cutting along the $m$-th coordinate, we say that we have completed a single iteration and we have a total of $(2)^m$ new pieces with particular values of the $m$-dimensional volumes. Again, appealing to Figure~\ref{fig:3dfragmentationunrestricted}, after cutting along the $x_3$ direction, we completed a single iteration and were left with $2^{3} \ = \ 8$ new volume pieces. For the second iteration, we repeat this entire process: we cut along the $x_1$ coordinate using a cut proportion $p_{2^{m}}$ and continue the process. We observe that, after $N$ iterations, we have $2^{Nm}-1$ cutting proportions using this numbering system which is reflected in our definitions of the variables used in the statement of Theorem~\ref{thm:unrestrictedmdim}. Now, we begin our proof of Theorem~\ref{thm:unrestrictedmdim} by first looking at the expectation.

\subsection{Expected value}\label{subsec:expectationunresmdim}

When calculating the expected value, the key insight is to observe that, effectively, we are finding the ``average'' 
of the significand indicator function: we are summing over $(2^m)^N$ terms and then dividing by this exact same number of terms: 

\bea\label{eqn:expectationvaluemdimfrag}
\mathbb{E}[P_N(s)] \ = \  \mathbb{E}\left[ \frac{ \sum _{i = 1} ^{(2^m)^N} \varphi_s(V_i)}{(2^m)^N} \right] \ = \  \frac{1}{(2^m)^N} \sum_{i = 1} ^{(2^m)^N} \mathbb{E} [\varphi_s(V_i)]. 
\eea
The tricky part is to figure out how to deal with this $V_i$, but the beautiful part is that we roughly know the nature of $V_i$. Essentially, this $V_i$ will always be a product of a certain $k$ number of factors $p_i$ taken from a distribution $f(t)$ satisfying the Mellin condition of Equation~\eqref{eqn:mellincondition}, whereas it will also have a total of $Nm - k$ proportions of the form $(1-p_i)$ which are taken form a distribution of the form $f(1-t)$ and then rescaled by some initial value $V$. Note that $t$ in this case is playing the role of a dummy variable: we are taking a value between $(0,1)$ for the proportion along every different coordinate for every volume piece. However, for our current purposes, we don't need to worry about this too much. Taking all of this into account, and relabeling when necessary, we conclude that $V_i$ can be expressed as
\bea\label{eqn:V_i}
    V_i \ = \  V \prod_{r = 1}^{k} p_r \prod_{j=k+1}^{Nm} (1-p_j).
\eea
By construction, Equation~\eqref{eqn:V_i} represents the product of independent random variables. While there are dependencies amongst the different volume pieces, the density cuts are chosen in such a way that they are independent of each other, which is important because it allows us to use the machinery of Theorem~\ref{thm:jkkkmmain}.

Using the definition of the expected value, we calculate the expectation of the significand indicator function by multiplying the significand indicator function of the $m$-dimensional volume $V_i$ with the associated probability density function and then integrating accordingly. Again, from the independence of the density cuts, we conclude that the probability of that particular volume piece $V_i$ occurring is equal to the product of all the density cuts that were involved in its creation. Hence, this will add an additional product of $k$ number of factors of the form $p_i$ taken from a density $f(t)$ and a total of $Nm - k$ proportions of the form $(1-p_i)$ which are taken from the density $f(1-t)$.

\begin{rmk}\label{rmk:expectationnote}
W are looking at the indicator function of this random variable, as this helps us build intuition for the proof.
\end{rmk}

As we only care about this problem as the number of iterations goes to infinity, and remembering that we are dealing with the uniform density, we get the following integral:

\bea \label{eqn:integralmdimexpectation}
\mathbb{E}[\varphi_s(V_i)] & \ = \ & \int_{p_1 = 0}^{1} \int_{p_2 = 0}^{1} \dotsm \int_{p_{Nm} = 0}^{1} \varphi_s \left( V \prod_{r = 1}^{k} p_r \prod_{j=k+1}^{Nm} (1-p_j) \right)\nonumber\\ & & \ \ \ \ \cdot \ \prod_{r=1}^{k} f(p_r) \prod_{j=k+1}^{Nm} f(1-p_j) \ dp_1 dp_2 \dotsm dp_{Nm}.
\eea
In Equation~\eqref{eqn:integralmdimexpectation}, we are studying the product of $Nm$ independent random variables, the density cuts, which are being rescaled by some initial volume $V$. From here, we can directly appeal to Theorem~\ref{thm:productisbenfofnindep} and be done with the problem, but it is worthwhile to explore some details about convergence\footnote{Many details are explored in Appendix A$.2$ of~\cite{BeckA}. The authors chose to prove some of their claims by appealing to familiar results about the Riemann-Zeta function, but the claims could have been obtained by also performing the calculations directly for the case of the uniform $(0,1)$ density.}. We know that this integral results in a $ \log_{10}(s)$ term along with an additional term: we must show that this term is bounded and that it decays as $N \to \infty$.
\begin{rmk}\label{rmk:noteaboouttendingtoinfinity}
Whenever we write $N \to \infty$, it is also implied that $Nm \to \infty$ since $m$ is a non-zero natural number. While it is useful to think about $Nm \to \infty$, we often omit the $m$ because we care about the fact that the number of iterations of the process tends to infinity which results in an infinitely large number of volume pieces. 
\end{rmk}
The decay follows from the fact that the Mellin transforms, whenever $\ell \ \neq \ 0$, are smaller than $1$ in absolute value. In the case when $\ell \ = \ 0$, we obtain our main contribution which is further explored in some detail in Appendix A$.2$ of~\cite{BeckA}. However, we can go a bit further.




As we have previously noted, a product of density functions which are symmetric around $1/2$ tends to $0$ because every individual term of the product tends towards $0$. For a function that is not symmetric around $1/2$, we can symmetrize it. Hence, effectively, any product of this form satisfying the Mellin transform condition can be reduced to a product of symmetric distributions with values which are all tending towards $0$. Hence, using Lemma~\ref{rmk:steve'sideaforsymmetry} we are making a gain in the sense that, instead of having to look at two functions and trying to make sure that both of them behave in a nice way, we have reduced ourselves to looking at one type of function. In the context of the fragmentation process, this just means that we are effectively always looking at the ``leftmost'' piece in every branch (refer to Figure~\ref{fig:3dfragmentationunrestricted}). This idea becomes increasingly important in our proof of the variance later on in Section~\ref{subsec:varianceunresmdim}.

From there, it follows that the base $10$ logarithms of the random variable are equdistributed on the appropriate interval, which means

\bea\label{eqn:expectationresult}
\lim_{N \to \infty}\mathbb{E}[P_N(s)]   \ = \  \log_{10}(s).
\eea

\begin{rmk}
We diverged a bit from the proof performed in~\cite{Beck}: the authors used the pigeonhole principle to prove this convergence. We could have also used the same trick, but we would have to add the $m$ factor to take into account the nature of the higher-dimensional fragmentation.
\end{rmk}

\subsection{Variance}\label{subsec:varianceunresmdim}

Even though the expected value converges to Benford, this is still not a guarantee that our process always remains Benford: we must show that the variance converges to $0$ in order to prove that the process becomes Benford as the number of iterations tends to infinity.
The most convenient form for the variance, in the case of our problem, is the following: 
\bea\label{eqn:varianceformulamdim}
{\rm Var}\left(P_N(s)\right) & \ = \ &  \mathbb{E}[(P_N(s))^2] - (\mathbb{E}[P_N(s)])^2,
\eea
where we added parentheses here for emphasis, but we usually omit them.
Before we begin our derivation, it is worth trying to simplify every little detail we can. One thing we observe is that we have some information about the square of the indicator function: it is exactly the same as the original indicator function. Hence, we can make our calculations easier by noting that
\bea\label{eqn:indicatorsame}
(\varphi_s(X_i))^2 \ = \ \varphi_s(X_i).
\eea
Using Equations~\eqref{eqn:varianceformulamdim} and~\eqref{eqn:expectationvaluemdimfrag}, we get the following:
\bea 
{\rm Var}\left(P_N(s)\right) & \ = \ & \mathbb{E}[P_N(s)^2] - \mathbb{E}[P_N(s)]^2  \label{eqn:variancealgebra1} \\
& \ = \ & \mathbb{E}\left[\left(\frac{ \sum _{i = 1} ^{(2^m)^N} \varphi_s(V_i)}{(2^m)^N} \right)^2\right] -  \mathbb{E}[P_N(s)]^2 \label{eqn:variancealgebra2} \\ 
& \ = \ & \mathbb{E}\left[ \frac{\sum_{i = 1}^{(2^m)^N} \varphi_s(V_i)^2}{(2^m)^{2N}} + \frac{\sum\limits_{\substack{i,j = 1 \\ i \neq j}}^{(2^m)^N} \varphi_s(V_i) \varphi_s(V_j)}{(2^m)^{2N}} \right]  -  \mathbb{E}[P_N(s)]^2  \label{eqn:variancealgebra3}\\ 
& \ = \ & \frac{1}{(2^m)^{2N}}\mathbb{E}\left[\sum_{i = 1}^{(2^m)^N} 
\varphi_s(V_i)^2 \right] + 
\frac{1}{(2^m)^{2N}}\mathbb{E}\left[\sum\limits_{\substack{i,j = 1 \\ i 
\neq j}}^{(2^m)^N} \varphi_s(V_i) \varphi_s(V_j) \right] \nonumber \\
& & \ \ \ \ \ \ \ \ \ \ \ \ \ \   -  \mathbb{E}[P_N(s)]^2
\label{eqn:variancealgebra4}\\ 
& \ = \ & \frac{1}{(2^m)^N} \mathbb{E}\left[P_N(s)\right]  + 
\frac{1}{(2^m)^{2N}}\mathbb{E}\left[\sum\limits_{\substack{i,j = 1 \\ i 
\neq j}}^{(2^m)^N} \varphi_s(V_i) \varphi_s(V_j) \right]  -  \mathbb{E}[P_N(s)]^2.\label{eqn:variancealgebra5}
\ \ \ \ \ \ \ \
\eea
We explain each step. In 
Equation~\eqref{eqn:variancealgebra2}, we directly applied the 
definition of the variance and plugged in the expression for the 
expected value. Then, we wrote out the result of squaring the function 
in Equation~\eqref{eqn:variancealgebra3} and we finally used the linearity 
of expectation to obtain Equation~\eqref{eqn:variancealgebra4}. 
Using the fact from Equation~\eqref{eqn:indicatorsame}, we recognize that 
we actually have Equation~\eqref{eqn:expectationvaluemdimfrag} embedded in 
this expression which we use to further simplify and we finally get the 
result in Equation~\eqref{eqn:variancealgebra5}. We obtained an 
expression for the expectation in Section~\ref{subsec:expectationunresmdim} which we now use. 

Noting again that the rate of convergence depends on the number of iterations as well as the dimension of the volume in question
we can say that
\bea\label{eqn:expvalmdimerrorfuncadd}
\mathbb{E}[P_N(s)] \ = \ \log_{10} s + o(1).
\eea

We plug Equation~\eqref{eqn:expvalmdimerrorfuncadd} directly into Equation~\eqref{eqn:variancealgebra5} and after the algebra obtain
\begin{align}
{\rm Var}\left(P_N(s)\right) & \ = \   \frac{1}{(2^m)^N} (\log_{10} s + o(1))  + 
\frac{1}{(2^m)^{2N}}\mathbb{E}\left[\sum\limits_{\substack{i,j = 1 \\ i 
\neq j}}^{(2^m)^N} \varphi_s(V_i) \varphi_s(V_j) \right] \nonumber \\
&  \ \ \ \ \ \ \ \ \ \ \ \ \ \ - (\log_{10} s + o(1))^2 \nonumber \\ \label{eqn:variancalg1} \\
&\ = \   \frac{1}{(2^m)^{2N}}\mathbb{E}\left[\sum\limits_{\substack{i,j = 1 \\ i 
\neq j}}^{(2^m)^N} \varphi_s(V_i) \varphi_s(V_j) \right]  - \log_{10}^2 s + o(1) \label{eqn:variancalg2},
\end{align}

where, in going from Equation~\eqref{eqn:variancalg1} to Equation~\eqref{eqn:variancalg2}, we used the fact that the first term tends to $0$ as the number of iterations goes to $\infty$ so we put it together with $o(1)\cdot\log_{10} s$ and $(o(1))^2$ and we simply count it all as an additional $o(1)$ term in the expression. 

We have now simplified the variance to a somewhat familiar problem; we write out the double sum in order to make some details more explicit, and again appeal to the linearity of expectation: 

\bea\label{eqn:doublesumvarworkedout}
{\rm Var}\left(P_N(s)\right) & \ = \ & 
\frac{1}{(2^m)^{N}}\frac{1}{(2^m)^{N}}\sum\limits_{\substack{i= 1 }}^{(2^m)^N}\sum\limits_{\substack{j= 1 
\\ j \neq i}}^{(2^m)^N}\mathbb{E}\left[\varphi_s(V_i) \varphi_s(V_j) 
\right]  - \log_{10}^2 s + o(1).
\eea
When written in this form, we can see that Equation~\eqref{eqn:doublesumvarworkedout} seems like some sort of an average that is being adjusted by the term $- \log_{10}^2 s + o(1)$. Essentially, our problem is reduced to evaluating $\mathbb{E}\left[\varphi_s(V_i) \varphi_s(V_j) 
\right]$. We prove that this term, in the appropriate limit, converges to $\log_{10}^2 s$. Now comes the difficult part of this calculation: evaluating the cross terms\footnote{As an example of a simpler variation of the ideas used here, see Exercise $3.4.1$ in~\cite{GS}.} over all $i \neq j$.

Before proceeding with the analysis, it is instructive to look at Figure~\ref{fig:3dfragmentationunrestricted}. Even though this figure does not represent the most general case, it is still showing us the basic spirit of this process: different volume pieces are being produced from an original piece (the $V$ that is present in all of them) and a certain number of independent random variables (the density cuts). Now, depending on the ``location'' in this tree, two different pieces can share more or less of these density cuts. It turns out that the number of density cuts that two pieces have in common becomes crucial in further analysing the variance.

We focus on a pair of pieces denoted as $(V_i,V_j)$ such that they share $M$ terms (density cuts) of the same form. In the future, we refer to these density cuts simply as ``factors'' since every volume piece is effectively a product of density cuts that are being rescaled by the initial volume $V$. After the $M$-th factor, $V_i$ and $V_j$ diverge from each other in such a way that $V_i$ has the density cut $p_{M+1}$ whereas the piece $V_j$ has the density cut $(1-p_{M+1})$. Finally, the leftover $Nm - M - 1$ terms in each product are independent from one another.
\begin{rmk}\label{rmk:subtlenoteaboutsharingMpieces}
There is a subtle note about the volume pieces having $M$ factors in common. Throughout this paper, we relabel the volume pieces in such a way that the first $M$ factors of the pair $(V_i,V_j)$ are the ones they have in common. We are allowed to do this because the order in which the factors appear has no impact on our analysis since we only care about the type of the factor. 
\end{rmk}
Without loss of generality, and relabelling when necessary according to the note from Remark~\ref{rmk:subtlenoteaboutsharingMpieces}, we can say that $V_i$ and $V_j$ are of the following form:

\bea\label{eqn:vivj} 
V_i  & \ = \ &  V \cdot p_1 \cdot p_2 \dotsm p_M \cdot p_{M+1} \cdot p_{M+2} \dotsm p_{Nm} \nonumber\\
V_j & \ = \ & V \cdot p_1 \cdot p_2 \dotsm p_M \cdot (1 - p_{M+1}) \cdot \tilde p_{M+2} \dotsm \tilde p_{Nm},
\eea 
where the tilde purely serves to emphasise where the divergence between the pieces occurs. Noting that we again take the density cuts from $f(t)$, as in the derivation in Section~\ref{subsec:expectationunresmdim}, and using the definitions of the pair $(V_i,V_j)$ from~\eqref{eqn:vivj}, we get the following:

{\small
\begin{align}\label{eqn:crossterms}
\mathbb{E}[\varphi_s(V_i)\varphi_s(V_j)] &\ = \    \int_{p_1=0}^1 \int _{p_2=0}^1 \dotsm \int _{p_{Nm}=0}^1 \int _{\tilde p_{M+2}=0}^1 \dotsm \int _{\tilde p_{Nm}=0}^1 \varphi _s \left(  V \prod _{r=1}^{M} p_r \cdot p_{M+1} \cdot \prod _{r = M+2} ^{Nm} p_r \right)  \nonumber \\
&  \ \ \ \ \cdot \  \varphi _s\left( V \prod _{r=1}^{M} p_r \cdot \left( 1-p_{M+1} \right) \cdot \prod _{r = M+2} ^{Nm} \tilde p_r \right) \nonumber \\
&  \ \ \ \ \cdot \  \prod _{r=1} ^{Nm} f (p_r)  \prod _{r=M+2} ^{Nm} f (1-\tilde p_r) \  dp_1 dp_2 \cdots dp_{Nm} d\tilde p_{M+2}\cdots d\tilde p_{Nm}.
\end{align}}
\noindent The integral in Equation~\eqref{eqn:crossterms} is tricky, but there is a useful idea we can use to approximate it, as we don't need an exact answer: we just need to make sure the answer is in the ballpark of the result we are looking for. Thinking about Figure~\ref{fig:3dfragmentationunrestricted} again, we can imagine this process continuing for many iterations. In such a case, having common terms amongst the pieces would be rather unlikely, which means that, at some point, the common terms become irrelevant. Essentially, we can then treat $\varphi_s(V_i)$ and $\varphi_s(V_j)$ as independent random variables so the expectation of their product is equal to the product of their expectations. Below we give the flavor of this idea and then we make it all rigorous later on. 
First split up the expectation of Equation~\eqref{eqn:crossterms} into two expectations (integrals) which can be approximated and then multiplied. Define the relevant volumes

\bea \label{eqn:steptowardsvarianceintegral}
V_1 \ :=\ V\left(\prod_{r=1}^M p_r\right) p_{M+1}, \ \ \ \ \ \ 
V_2 \ :=\ V\left(\prod_{r=1}^M p_r\right) (1-p_{M+1}).
\eea 

\begin{rmk}
Equation~\eqref{eqn:crossterms} is missing a certain number of density cuts: this is very much intentional and will become clear by the end of Section~\ref{subsec:errortermbound}.
\end{rmk}
Following the heuristic approach proposed by Becker et al.~\cite{Beck}, we can construct the integrals

\bea \label{eqn:varintegrals}
I(V_1) &  :=  & \int _{ p_{M+1} = 0} ^{1} \cdots \int _{ p_{Nm} = 0} ^{1} \varphi _s \left( V_1 \prod _{r = M+1} ^{Nm} p_r \right) \prod _{r = M+1} ^{Nm} f (p_r) \   dp_{M+1}dp_{M+2}\cdots dp_{Nm} \nonumber\\
J(V_2) & \ :=\  & \int _{ \tilde p_{M+1} = 0} ^{1} \cdots \int _{\tilde p_{Nm} = 0} ^{1} \varphi _s \left( V_2 \prod _{r = M+1} ^{Nm} \tilde p_r \right) \prod _{r = M+1} ^{Nm} f (\tilde p_r) \  d\tilde p_{M+1}d \tilde p_{M+2}\cdots d \tilde p_{Nm}.\nonumber\\
\eea
However, observe that we have a product of independent random variables in Equation~\eqref{eqn:varintegrals} which, by our previous work, converges to Benford behaviour.

This indeed implies that $$|I(V_1)J(V_2)-(\log_{10} s)^2| \ = \ o(1),$$ so the variance goes to zero as was expected. We make all of this rigorous below. 

\subsection{Removing terms}\label{sec:removingterms}
The biggest issue we must deal with in evaluating the variance is the problem of the cross-terms. By looking at Figure~\ref{fig:3dfragmentationunrestricted}, we observe that the different volume pieces must depend on each other because we can always find some sort of a path to come back to a volume piece which is the ``parent'' of some two volumes $V_i$ and $V_j$. Hence, we have to keep track of these dependencies, but our claim is that certain dependencies matter more than others. The key question is, when we choose two pieces at random on some level $\nu \leq N$, is it more likely that they have a lot of terms in common or only a few terms in common? 

First, observe from Figure~\ref{fig:3dfragmentationunrestricted} that volume pieces always have the initial volume $V$ in common, as well as a certain number of factors of the form $p_k$ and $1-p_k$. As always, note that we are using the representation for the $m \ = \ 3$ case to make the visualization somewhat easier, but this extends for any finite-dimensional volume $m$. With that being said, we note that it is always possible to relabel the density cuts in order to highlight the differences between the volume pieces, which is exactly what we did in defining the pair $(V_i,V_j)$:

\bea\label{eqn:vivjagain}
V_i  & \ = \ &  V \cdot p_1 \cdot p_2 \dotsm p_M \cdot p_{M+1} \cdot p_{M+2} \dotsm p_{Nm} \nonumber\\
V_j & \ = \ & V \cdot p_1 \cdot p_2 \dotsm p_M \cdot (1 - p_{M+1}) \cdot \tilde p_{M+2} \dotsm \tilde p_{Nm}.
\eea 

We consider the dependencies and the likelihood of $V_i$ and $V_j$ having $M$ terms in common. When we look at the branching process in Figure~\ref{fig:3dfragmentationunrestricted}, we notice that at every step, we either apply a density cut of the form $p_i$ or of the form $1-p_i$. Hence, at every step, there is a $50 \%$ chance of having more than a certain number of cuts in common. At the next step, it is $50 \%$ of that previous probability and we effectively end up having some sort of a geometric sequence converging to $0$ in the appropriate limit. The probability that the two randomly chosen volume pieces $(V_i,V_j)$ have exactly $M$ factors in common is $2^{-M}$. However, the probability of them agreeing to up to \emph{at most} $M$ factors is $1-2^{-M}$.
\begin{rmk}\label{rmk:whynom}
The $m$, denoting the dimension of the volume we are working with, doesn't show up in our analysis, but this $m$ is implicitly embedded into the structure of $M$ because we are focusing on the density cuts that are present in the different volume pieces. To make this clearer, we remind the reader that $1 \leq M \leq mN$ in general, but we develop a bound for $M$ shortly in order to guarantee convergence. 
\end{rmk}
This means that, for an appropriately chosen $M$, we can completely ignore volume pieces that have a very high dependency, as expressed through a large number of density cuts that they have in common. We effectively want the term $2^M$ to go off to infinity, and we can easily achieve this as long as $M(N) \ = \ o(N)$. It turns out that any function of $M$ in terms of $N$ that doesn't grow as fast as the linear function $g(N) \ = \ N$ is a good candidate. This insight allows us to treat the integrals $I(V_1)$ and $J(V_2)$ as essentially independent in the probabilistic sense. Hence, we choose to have $M(N) \ = \ \log N$. We extend our argument even further by looking at the double sum in Equation~\eqref{eqn:doublesumlook}:
\bea\label{eqn:doublesumlook}
{\rm Var}\left(P_N(s)\right) & \ = \ & 
\frac{1}{(2^m)^{N}}\frac{1}{(2^m)^{N}}\sum\limits_{\substack{i= 1 
}}^{(2^m)^N}\sum\limits_{\substack{j= 1 
\\ j \neq i}}^{(2^m)^N}\mathbb{E}\left[\varphi_s(V_i) \varphi_s(V_j) 
\right]  - \log_{10}^2 s + o(1).
\eea
The expected value in the double sum is bounded by $1$ so we conclude that the contribution of the expected value from these volume pieces with high dependencies tends to $0$ as $M(N) \to \infty$. This allows us to remove these terms from our consideration and to include them in the $o(1)$ term. We now move on to discussing the bounds of these error terms and a clever application of the triangle inequality.

\subsection{Bounding the error terms}\label{subsec:errortermbound}

We know that every individual piece is actually bounded, but we are concerned that something in the integration might blow up as the number of pieces becomes infinitely large. That ``something'' is an infinite sum of error terms. 
This is why we need to make sure that the errors stemming from the integrals in Equation~\eqref{eqn:varintegrals} are sufficiently bounded so that they disappear when divided by the total number of pieces. 

From Equation~\eqref{eqn:boundcondition}, we know that the probability of the logarithmically rescaled random variable is bounded\footnote{Note that using the logarithmic rescaling in our results makes sense twofold: not only because of its natural use in Benford's law, but also because we are dealing with the uniform $(0,1)$ distribution}. Using that condition, we can explicitly write out a bound for $I(V_1)$ and $J(V_2)$, which we make precise in Lemma~\ref{lemma:agreeuptoatmostMlevels}.

We first apply Theorem~\ref{thm:jkkkmmain} to Equation~\eqref{eqn:varintegrals} to get

\bea\label{eqn:bound}
\abs{I(V_1) - \log_{10}s} \ & \leq & \  (b-a) \cdot \sum_{\ell = -\infty \atop \ell \neq 0}^\infty \abs{\prod_{k=M+1}^{Nm} \mathcal{M} f_{k}\left(1-\frac{2\pi i \ell}{\log B}\right)} \nonumber \\
\abs{J(V_2) - \log_{10}s} \ & \leq & \  (b-a) \cdot \sum_{\ell = -\infty \atop \ell \neq 0}^\infty \abs{\prod_{k=M+1}^{Nm} \mathcal{M} f_{k}\left(1-\frac{2\pi i \ell}{\log B}\right)}. 
\eea
We now recall our original Mellin Transform condition, as expressed in Theorem~\ref{thm:productisbenfofnindep}:

\bea\label{eqn:mellintransusedforbound}
\lim_{N\to\infty} \sum_{\ell = -\infty \atop \ell \neq 0}^\infty \abs{\prod_{k=1}^N \mathcal{M}f_{k}\left(1-\frac{2\pi i \ell}{\log 10}\right)} \ = \ 0.
\eea
Because the expression in~\eqref{eqn:mellintransusedforbound} converges, we can define a bound $D_{T(M)}$, where $T(M) \ \coloneqq \ Nm - (M+1)$, by
\bea\label{eqn:boundD}
D_{T(M)} \ = \ (b-a) \cdot \sum_{\ell = -\infty \atop \ell \neq 0}^\infty \abs{\prod_{k=1}^{T(M)} \mathcal{M} f_{k}\left(1-\frac{2\pi i \ell}{\log B}\right)},
\eea
which allows us to rewrite the Expressions in~\eqref{eqn:bound} as

\bea\label{eqn:Dt}
\abs{I(V_1) - \log_{10}s} \ & \leq & \ D_{T(M)} \nonumber \\
\abs{J(V_2) - \log_{10}s} \ & \leq & \ D_{T(M)}. 
\eea

\noindent 
Observe that we use Equation~\eqref{eqn:Dt} only when we agree up to at most $M$ factors, which means we have already removed the pairs with high dependency since
we know that those pieces are bounded with at most $1$ so they go to zero when divided by the total number of pieces. Furthermore, observe the dependency of $D_{T(M)}$ on $M$ since $T$ is a fixed number depending on our choice of terms that are in common amongst the pairs. For completion, we also remind the reader that the key idea of the proof is that $M$ grows rapidly, but still slower than $N$, so a choice of $M \ = \ \log N$ works. Using the triangle inequality, we can write
\bea\label{eqn:vartrig}
|I(V_1)J(V_2)-(\log_{10} s)^2| \ & = & \ \abs{I(V_1)J(V_2) - J(V_2)\log_{10}s + J(V_2)\log_{10}s - (\log_{10} s)^2} \nonumber \\
& \leq & \ \abs{I(V_1)J(V_2) - J(V_2)\log_{10}s} + \abs{J(V_2)\log_{10}s - (\log_{10} s)^2} \nonumber \\
& \leq & \ \abs{I(V_1)J(V_2) - J(V_2)\log_{10}s} + \abs{J(V_2)\log_{10}s - (\log_{10} s)^2} \nonumber \\
& \leq & \ \abs{J(V_2)}\abs{I(V_1) - \log_{10}s} + \abs{\log_{10}s}\abs{J(V_2) - \log_{10} s} \nonumber \\
& \leq & \ \abs{I(V_1) - \log_{10}s} + \abs{J(V_2) - \log_{10} s},
\eea
where in the last line we used the fact that $\abs{J(V_2)}$ and $\abs{\log_{10}s}$ are both bounded by $1$. Hence, it follows that

\bea\label{eqn:trigineqbound}
\abs{I(V_1)J(V_2)-(\log_{10} s)^2} \ & \leq & \ 2D_{T(M)}. 
\eea

\noindent We summarize all of this in Lemma~\ref{lemma:agreeuptoatmostMlevels}.

\begin{lem}[Bounding the error term]\label{lemma:agreeuptoatmostMlevels}
Assuming the Mellin condition holds, there are two distinct possible cases for the bound expressed in Equation~\eqref{eqn:trigineqbound}.
\begin{itemize}
    \item In the case of high dependence, meaning that $V_i$ and $V_j$ have many factors in common so they agree to more than $M$ factors, Equation~\eqref{eqn:trigineqbound} has the trivial bound $1$. 
    \item In the case of low dependence, meaning that $V_i$ and $V_j$ have few terms in common so we agree up to at most $M$ factors, Equation~\eqref{eqn:trigineqbound} is unchanged, hence the difference is bounded by $D_{T(M)}$.
\end{itemize}
\end{lem}

\noindent Our goal is to estimate

\bea\label{eqn:morebound}
\abs{\frac{1}{(2^m)^{N}}\frac{1}{(2^m)^{N}}\sum\limits_{\substack{i= 1 
}}^{(2^m)^{N}}\sum\limits_{\substack{j= 1 
\\ j \neq i}}^{(2^m)^{N}}\mathbb{E}\left[\varphi_s(V_i) \varphi_s(V_j) 
\right] - \log^2_{10} s },
\eea

\noindent and we note that the insight from Equation~\eqref{eqn:trigineqbound} and Lemma~\ref{lemma:agreeuptoatmostMlevels} will be crucial, along with the conversation around Equation~\eqref{eqn:doublesumvarworkedout}. The biggest problem we have is that it becomes somewhat difficult to deal with all the cases when $j \ \neq i$, but we can resolve this issue by deeply thinking about the impact these cases have on our analysis. Let us remove this restriction for the $j$ and observe that, for a fixed $i$, the number of indices $j$ whose factors agree with those from $i$ for at least $M+1$ factors is $2^{mN}/2^{M+1}$. Since the summand is bounded by $1$ and we now divide with the total number of pairs, which is $2^{2mN}$, we see this term is of the form $O(1/2^{M}) \ = \ o(1)$ since $M \ = \ \log N$ which tends to $0$ as $N \to \infty$. Hence, this can be absorbed into the error term already present in Equation~\eqref{eqn:doublesumvarworkedout}. Now, we have reduced our problem to evaluating 
\bea\label{eqn:removedthej}
\abs{\frac{1}{(2^m)^{N}}\frac{1}{(2^m)^{N}}\sum\limits_{\substack{i= 1 
}}^{(2^m)^{N}}\sum\limits_{\substack{j= 1 }}^{(2^m)^{N}}\left(\mathbb{E}\left[\varphi_s(V_i) \varphi_s(V_j) 
\right] - \log^2_{10} s\right) },
\eea
where the major change is that we removed the restriction in Equation~\eqref{eqn:removedthej} which allowed us to put the logarithm into the double sum. We can now start thinking about the pieces $V_i$ and $V_j$ somewhat more directly. Our goal is to break up that double sum into different sums, based on the level of dependency between the pieces. We introduce a metric for the dependency through the variable $\mu$ which represents the number of terms $V_i$ and $V_j$ have in common with each other. This allows us to break up the double sum as 
\bea\label{eqn:breakingupdoublesum}
\abs{\frac{1}{(2^m)^{N}}\frac{1}{(2^m)^{N}}\sum_{i}\sum\limits_{\substack{\mu= 0 }}^{mN}\sum_{j_{\mu(i)}}\left(\mathbb{E}\left[\varphi_s(V_i) \varphi_s(V_j) 
\right] - \log^2_{10} s\right) },
\eea
where we note that $j$ depends on the number of factors $V_i$ and $V_j$ have in common which, naturally, depends on our choice of $V_i$. Furthermore, $j_{\mu(i)}$ is the set of indices $j$ such that $V_i$ and $V_j$ agree for $\mu$ factors.
The key is now to split the $\mu$ sum into two pieces based on the level of 
dependency between $V_i$ and $V_j$. When $0 \leq \mu \leq M$, this represents the case where $V_i$ and $V_j$ are virtually independent so we can directly appeal to Equation~\eqref{eqn:trigineqbound}. On the other hand, when $M+1 \leq \mu \leq Nm$, we are dealing with the case with high dependency so the error is bounded by $1$. We make the analysis more explicit by carefully presenting the two diffferent cases. 

\subsubsection{High dependence between pieces}\label{subsubsec:highdependence}
When the pieces have a high dependence, meaning that they agree to more than $M$ factors, then the difference in Equation~\eqref{eqn:trigineqbound} is bounded by $1$.
\bea\label{eqn:trigineqboundhighdep}
\abs{I(V_1)J(V_2)-(\log_{10} s)^2} \ & \leq & \ 1. 
\eea
From Section~\ref{sec:removingterms}, we know this is a rather unlikely case so the total number of such pairs is equal to $2^{2mN}/2^{M+1}$, which goes to zero when we divide it by the total number of pairs. Explicitly,
\bea\label{eqn:highdependence}
\abs{\frac{1}{(2^m)^{N}}\frac{1}{(2^m)^{N}}\sum_{i}\sum\limits_{\substack{\mu= M+1 }}^{mN}\sum_{j_{\mu(i)}}\left(\mathbb{E}\left[\varphi_s(V_i) \varphi_s(V_j) 
\right] - \log^2_{10} s\right)} \ &\leq& \ \abs{\frac{1}{(2^m)^{N}}\frac{1}{(2^m)^{N}} \frac{2^{2mN}}{2^{M+1}}} \nonumber \\
\ & = & \  \abs{\frac{1}{2^{M+1}}},
\eea
which tends to zero as $N \to \infty$ because we defined $M \ = \ \log N$.

\subsubsection{Low dependence between pieces}\label{subsubsec:lowdependence}
When the pieces have a low dependence, meaning that they agree to at most $M$ factors then the difference in Equation~\eqref{eqn:trigineqbound} is bounded by $D_{T(M)}$ and we can use it directly:
\bea\label{eqn:trigineqboundlowdep}
\abs{I(V_1)J(V_2)-(\log_{10} s)^2} \ & \leq & \ D_{T(M)} . 
\eea
From Section~\ref{sec:removingterms}, we know this case is likely. The probability of this happening is $1-1/2^{M+1}$, so the total number of such pairs is $2^{2mN}-2^{2mN}/2^{M+1}$. Notice that this is exactly the same number we would have gotten by applying the logic that we are summing up over the leftover pieces, after removing the pieces with high dependency, i.e., $2^{2mN}-2^{2mN}/2^{M+1}$. This expression also goes to zero when we divide by the total number of pieces. Explicitly,
{\fontsize{10.5pt}{10.8pt}
\begin{align}\label{eqn:lowdependence}
\abs{\frac{1}{(2^m)^N}\frac{1}{(2^m)^N}\sum_{i}\sum_{\substack{\mu=0}}^{M}\sum_{j_{\mu(i)}}\left(\mathbb{E}\left[\varphi_s(V_i) \varphi_s(V_j)\right] - \log^2_{10} s\right)} 
& \ \leq \ \abs{\frac{1}{(2^m)^N}\frac{1}{(2^m)^N}\left(2^{2Nm}-\frac{2^{2Nm}}{2^{M+1}}\right)D_{T(M)}} \nonumber \\
& \ = \ \abs{\left(1-\frac{1}{2^{M+1}}\right)D_{T(M)}},
\end{align}}

\noindent which tends to zero as $N \to \infty$.

\noindent Finally, based on Equation~\eqref{eqn:highdependence} and Equation~\eqref{eqn:lowdependence}, we conclude that, as $N \to \infty$,

\bea \label{eqn:morebound2}
\frac{1}{(2^m)^{N}}\frac{1}{(2^m)^{N}}\sum\limits_{\substack{i= 1 }}^{(2^m)^{N}}\sum\limits_{\substack{j= 1 
\\ j \neq i}}^{(2^m)^{N}}\mathbb{E}\left[\varphi_s(V_i) \varphi_s(V_j) 
\right] - \log^2_{10} s \ = \ 0.
\eea
Combining Equation~\eqref{eqn:boundcondition} and Equation~\eqref{eqn:doublesumvarworkedout} we get that

\bea\label{eqn:varianceresultfinally}
{\rm Var}\left(P_N(s)\right) & \ \leq \ & o(1),
\eea
which, by definition is equal to $0$ in the $N \to \infty$ limit.

Finally, combining Equation~\eqref{eqn:expectationresult} and Equation~\eqref{eqn:varianceresultfinally}, we conclude that our unrestricted $m$-dimensional fragmentation process does indeed result in Benford behavior.





\section{Exploring the Conjecture}\label{sec:conjecture}
Our expectation is that even a lower-dimensional volume would converge to Benford under this particular process. In our work, some simulations from Mathematica seem to show the convergence of the perimeter of new pieces stemming from a fragmentation process in a rectangle. While it seems that there is convergence of volumes of dimensions close to the highest dimension of the object, the question remains as to what happens to quantities whose dimension is substantially smaller than the highest dimension of the object being studied. 

The truly interesting question is to see how close the lower-dimensional volume and the higher-dimensional volume have to be (i.e., we want to know what impact the size of the difference between $d$ and $m$ has on the rate of convergence to Benford, where $d$ and $m$ are defined in Conjecture~\ref{conj:benfconjddimlessthanmdim}).


\section{Proposed Further Research Avenues}\label{sec:futurework}

The fragmentation process we looked at is defined as the unrestricted fragmentation process in the work of~\cite{Beck}. One potential research avenue would be to check whether or not the other described fragmentation processes can also be generalized and extended to $m$-dimensions. Furthermore, one could also pose the same question about the lower-dimensional volumes following the Benford distribution in this case. See~\cite{Betti} for some recent progress.


\end{document}